\documentclass[11pt]{article}
\usepackage{amsfonts}
\usepackage{amssymb}
\usepackage{graphics}
\usepackage{amsmath,epsfig,psfrag}
\usepackage[english]{babel}
\usepackage[all]{xy}
\usepackage{amscd}
\usepackage{amsthm}
\usepackage{latexsym}
\usepackage{graphicx}
\usepackage{mathrsfs}

\def\Z{{\Bbb Z}}

\newtheorem{theorem}{Theorem}[section]
\newtheorem{corollary}[theorem]{Corollary}
\newtheorem{lemma}[theorem]{Lemma}

\theoremstyle{definition}
\newtheorem{example}[theorem]{Example}

\theoremstyle{remark}

\begin{document}
\title{Embedding periodic maps on surfaces \\into those on $S^3$}

\author{Yu Guo, Chao Wang, Shicheng Wang, Yimu Zhang}

\date{}
\maketitle

\centerline{School of Mathematical Science} \centerline{Peking
University, Beijing, 100871 CHINA}


\begin{abstract} Call a periodic map $h$ on the closed
orientable surface $\Sigma_g$ extendable if $h$ extends to a
periodic map over the pair $(S^3, \Sigma_g)$ for possible embeddings
$e: \Sigma_g\to S^3$.

We determine the extendabilities for all periodical maps on
$\Sigma_2$. The results involve various orientation
preserving/reversing behalves  of the periodical maps on the pair
$(S^3, \Sigma_g)$. To do this we first list  all periodic maps on
$\Sigma_2$, and indeed we exhibit each of them  as a composition of
primary and explicit  symmetries, like rotations, reflections and
antipodal maps, which itself should be an interesting piece.

A  by-product is that for each even $g$, the maximum order periodic
map on $\Sigma_g$ is extendable, which contrasts sharply to the
situation in orientation preserving category.  \footnote{


 email: wangsc@math.pku.edu.cn}

\end{abstract}


\section{Introduction}
Closed orientable surfaces are most ordinary geometric and physical
subjects to us, since they stay in our 3-dimensional space
everywhere in various manners (often as the boundaries of
3-dimensional solids).

The study of symmetries on closed orientable surfaces is also a
classical topic in mathematics. An interesting fact is that some of
those symmetries are easy to see, and others are not, or more
precisely to say, some symmetries are more visible than others.

Let's have a look of examples. Let $\Sigma_g$ be the orientable
closed surface with genus $g$. We always assume that $g>1$ in this
note. It is easy to see that there is an order 2 symmetry $\rho$ on
$\Sigma_2$ indicated in the left side of Figure 1, and it is not
easy to see that there is a symmetry $\tau$ of order 2 on $\Sigma_2$
whose fixed point set consists of two non-separating circles,
indicated as the right side of Figure 1. A primary reason for this
fact
 is that we can embed $\Sigma_2$ and $\rho$ into the 3-space and
its symmetry space simultaneously, simply to say $\rho$ is induced
from a symmetry of our 3-space,  or $\rho$ extends to a symmetry
over the 3-space; and on the other hand, as we will see that $\tau$
can never be induced by a symmetry of the 3-space for any embedding
of $\Sigma_2.$

\begin{center}
\scalebox{0.7}{\includegraphics{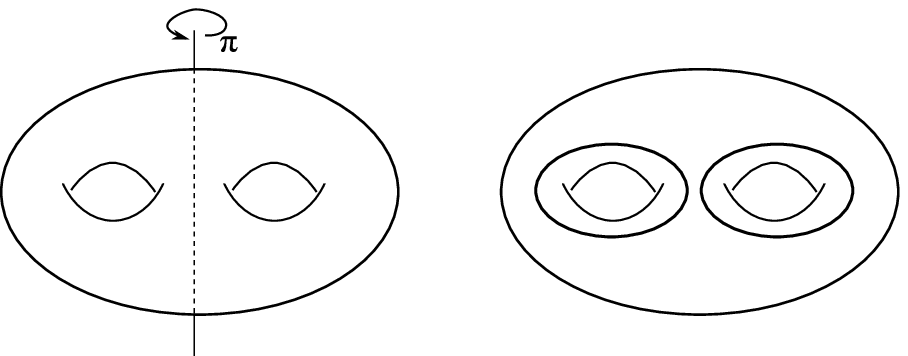}}

Figure 1
\end{center}

Now we make a precise definition: If a finite group action $G$ on
$\Sigma_g$ can also acts on the pair $(S^3, \Sigma_g)$ for possible
embeddings $e: \Sigma_g\to S^3$, which is to say, $\forall h \in G$
we have $h\circ e=e\circ h$, we call such a group action on
$\Sigma_g$ extendable over $S^3$ (with respect to $e$).

Such extendable finite group actions have been addressed in
\cite{WWZZ1} and \cite{WWZZ2}. But those two papers focus on the
maximum orders problem in the orientation preserving category. In
this note we start to discuss the extendable problem for general
finite group actions and  we will delete the orientation preserving
restriction. We will first focus on the simplest case: the cyclic
group actions on the surface $\Sigma_2$.

We examine the extendabilities for all periodical maps on
$\Sigma_2$. To do this we first need not only to exhibit all
 $\Z_n$-actions on $\Sigma_2$, but also exhibit them in a very geometric and
 visible  way.

 \begin{theorem}\label{classifacation and extendibility}
(1) There are twenty one conjugacy  classes of finite cyclic actions
on $\Sigma_2$ which are generated by the following periodical maps:
$\rho_{2,1}$, $\rho_{2,2}$, $\tau_{2,1}$, $\tau_{2,2}$,
$\tau_{2,3}$, $\tau_{2,4}$, $\tau_{2,5}$, $\rho_3$, $\rho_4$,
$\tau_{4,1}$, $\tau_{4,2}$, $\rho_5$, $\rho_{6,1}$, $\rho_{6,2}$,
$\tau_{6,1}$, $\tau_{6,2}$, $\tau_{6,3}$, $\rho_8$, $\tau_8$,
$\rho_{10}$, $\tau_{12}$; where each map presented by $\rho$/$\tau$
is orientation preserving/reversing, and the first subscribe
indicates the order.


(2) The extendability of twenty one periodic maps  in (1) are:
$\rho_{2,1}\{+\}$, $\rho_{2,2}\{+,-\}$, $\tau_{2,1}\{+,-\}$,
$\tau_{2,2}\{+,-\}$, $\tau_{2,3}\{+\}$, $\tau_{2,4}\{\varnothing\}$,
$\tau_{2,5}\{-\}$, $\rho_3\{+\}$, $\rho_4\{-\}$,
$\tau_{4,1}\{+,-\}$, $\tau_{4,2}\{\varnothing\}$,
$\rho_5\{\varnothing\}$, $\rho_{6,1}\{+\}$, $\rho_{6,2}\{-\}$,
$\tau_{6,1}\{\varnothing\}$, $\tau_{6,2}\{-\}$,
$\tau_{6,3}\{\varnothing\}$, $\rho_8\{\varnothing\}$,
$\tau_8\{\varnothing\}$, $\rho_{10}\{\varnothing\}$,
$\tau_{12}\{+\}$; where the symbols $\{+\}$/ $\{-\}$/ $\{+,-\}$/
$\{\varnothing\}$ indicates the map has orientation
preserving/orientation reversing/both orientation preserving and
reversing/no extension.

A geometric description of (1) and (2) are given in Figure 2 and 3,
where we exhibit each of them as a composition of rotations,
reflections, and (semi-)antipodal maps:

(i) each rotation in 2-space (3-space) is indicated by a circular
arc with arrow around a point (an axis);

(ii) each reflection about a $2\text{-sphere} = 2\text{-plane} \cup
\infty$ or a $circle=line\cup \infty$ is indicated by arc with two
arrows,

(iii) each (semi-)antipodal map is indicated by a point (another
fixed point is $\infty$).

More concrete description of those maps will be given in the proof
of Theorem \ref{classifacation and extendibility} and Example
\ref{cage}.
\end{theorem}

\begin{center}
\scalebox{0.7}{\includegraphics{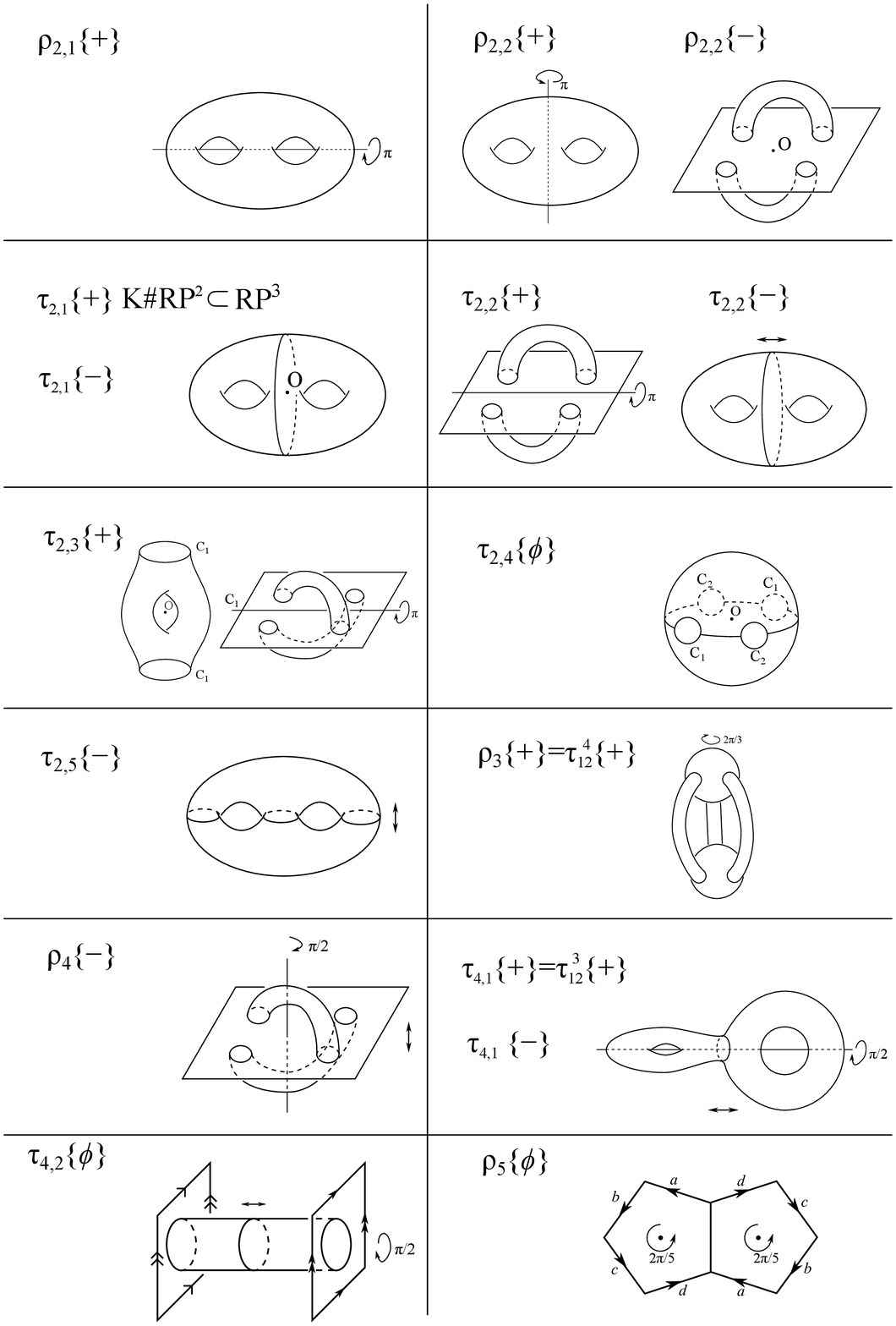}}

Figure 2
\end{center}

\begin{center}
\scalebox{0.75}{\includegraphics{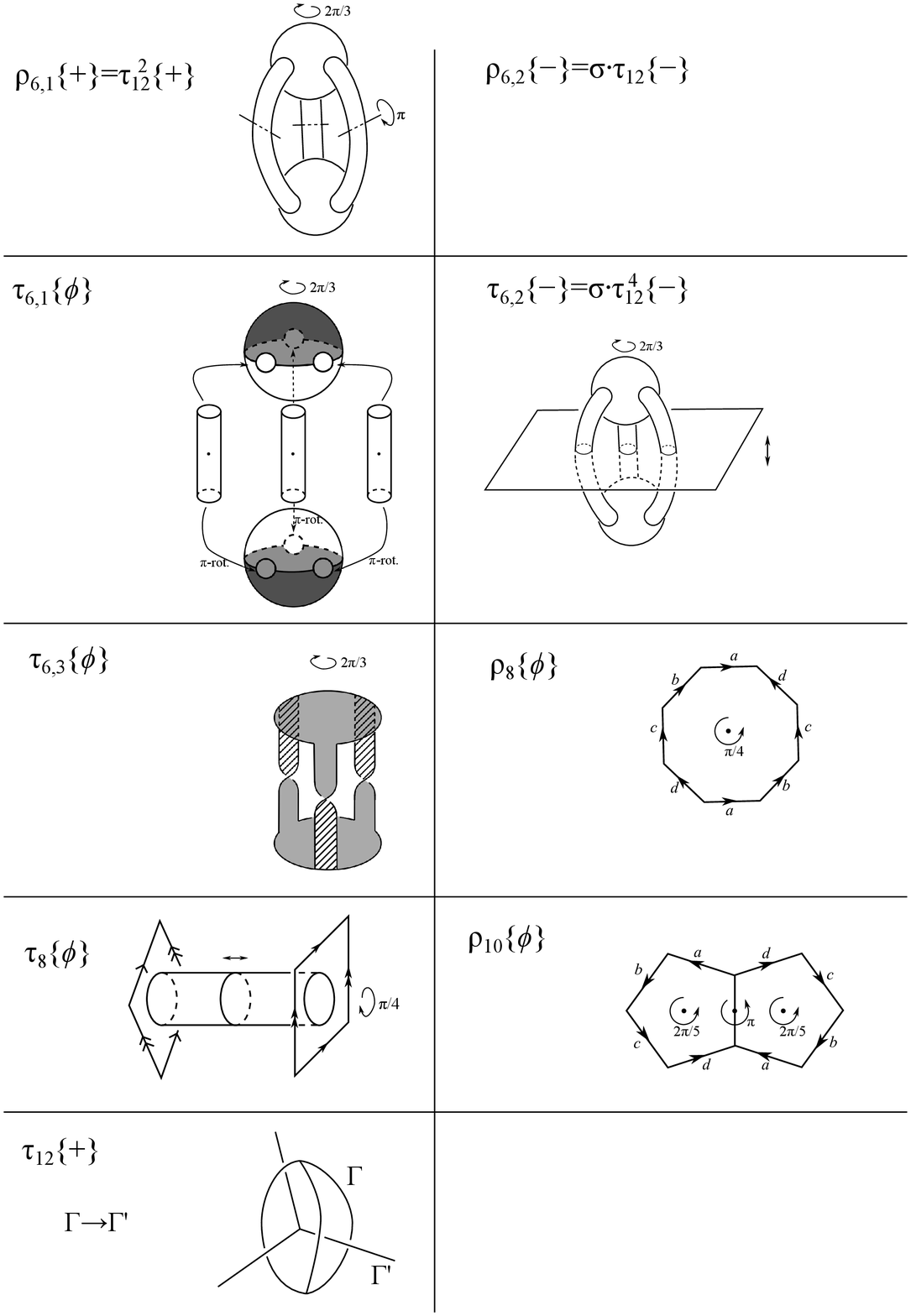}}

Figure 3
\end{center}
\newpage
Let $C_g$ and  $CE_g$ be the maximum orders of periodical maps and
of extendable periodical maps on $\Sigma_g$ respectively; $C_g^o$
and  $CE_g^o$ be the corresponding notions but restricted on
orientation-preserving category. Then it is known that (1) $C^o_g$
is $4g+2$. ${C}_g$ is $4g+4$ when $g$ is even and $4g+2$ when $g$ is
odd \cite{St}. (2) $CE^o_g=2g+2$ if $g$ is even and $2g-2$ if $g$ is
odd \cite{WWZZ1}.

By the construction and argument provided for the map $\tau_{12}$
and $\tau_{2,4}$ (see Example \ref{cage} and Case $(2_-)$), we are
easy to get the following facts, which  do not appear in the
orientation preserving category.

\begin{corollary}
(1) For each even $g$, the maximum order periodic map on $\Sigma_g$
is extendable, that is to say $CE_g=4g+4$.

(2) For each $g$, there is a non-extendable  order 2 symmetry on
$\Sigma_g$.
\end{corollary}

The following notions are convenient for our later discussion.

Let $G=\Z_n =\langle h\rangle$, where $h$ is a periodic map of order
$n$ on $(S^3, \Sigma_g)$. According to whether $h$
preserves/reverses the orientation of $S^3$/$\Sigma_g$, we have four
types of extendable group actions.

(1) Type $(+,+)$: $h$ preserves the orientations of both $S^3$ and
$\Sigma_g$.

(2) Type $(+,-)$: $h$ preserves the orientation of $\Sigma_g$ and
reverses that of $S^3$.

(3) Type $(-,+)$: $h$ reverses the orientation of $\Sigma_g$ and
preserves that of $S^3$.

(4) Type $(-,-)$: $h$ reverses the orientations of both $\Sigma_g$
and $S^3$.

One can check easily that the type of the action $G=\langle
h\rangle$ is independent of the choices of the periodical map $h$.
Those notions remind us that Type $(+,+)$ and Type $(+,-)$ are
extending orientation preserving maps on $\Sigma_g$ to those of
$S^3$ in orientation preserving and reversing way respectively, and
Type $(-,+)$ and Type $(-,-)$ are extending orientation reversing
maps on $\Sigma_g$ to those of $S^3$ in orientation preserving and
reversing way respectively. Periodical maps of Type $(+,+)$ and Type
$(-,-)$ do not change the two sides of $\Sigma_g$, and of the
remaining two types  do change the two sides of $\Sigma_g$.

Notice that if $G$ is extendable and $h \in G$ is an element which
equals identity on $\Sigma_g$, then it is easy to see that $h$ is
identity on the whole $S^3$. Hence we always assume that the group
action is faithful on both $\Sigma_g$ and $S^3$.

Suppose $S$ (resp. $P$) is a properly embedded $(n-1)$-manifold
(resp. $n$-manifold) in a n-manifold $M$. We use $M\setminus S$
(resp. $M\setminus P$) to denote the resulting manifold obtained by
splitting $M$ along $S$ (resp. removing $\text{int} P$, the interior
of $P$).

The fixed point set $Fix(G)$ for a finite group action on $X$ is
defined as $\{x\in X\mid g(x)=x, \forall g\in G\}$.

{\bf Content of the paper:} In Section \ref{2example}, we recall
some fundamental results, construct some examples, and prove some
lemmas, which will be used in Section \ref{Proof of the main
theorem} to prove Theorem \ref{classifacation and extendibility}.

\section{Some facts and examples\label{2example}}

We first recall some fundamental results in this section.

Consider 2-disc $D^2$, resp. 3-disc $D^3$ as the unit disc in
$\mathbb{R}^2$ resp. $\mathbb{R}^3$; and 2-sphere $S^2$, resp.
3-sphere $S^3$ as the unit sphere in $\mathbb{R}^3$, resp.
$\mathbb{R}^4$. Then each periodic map on $D^2$, $S^2$ and $B^3$,
$S^3$ can be conjugated into $O(2)$, $O(3)$, $O(4)$ respectively.
Those results in dimension 2 can be found in \cite{Ei} , and in
dimension 3 can be found in \cite{Li} for actions of isolated fixed
points, in \cite{Th} and \cite{BP} for actions with fixed point set
of dimension at least 1, in \cite{Per} for fixed point free actions.

There is exactly one standard  orientation-reversing $\Z_2$-action
on $D^2$:  a reflection about a diameter. There are two standard
orientation-reversing $\Z_2$-actions on $D^3$ (resp. $S^2$): one is
a reflection  about an equator 2-disc (circle), the other is the
antipodal map. There are also two standard orientation-reversing
$\Z_2$-actions on $S^3$ : one is a reflection about an equator
2-sphere, another is the semi-antipodal map which has two fixed
points $\{0,\infty\}$, and on every sphere $\{(x,y,z)\mid
x^2+y^2+z^2=R^2\}$ is an antipodal map.

The facts in the following statement, which can be found in
\cite{Ei}, \cite{Li}, \cite{Th}, \cite{BP}, \cite{Per} and
\cite{Sm},   will be used repeatedly later.

\begin{theorem}\label{fixed point set}
(1) Any orientation reversing periodic map on the 2-disc $D^2$ is
conjugate to a reflection about a diameter.

(2) An orientation reversing $\Z_2$-action on $S^3$($D^3$) must
conjugate to either a reflection  about a $2$-sphere (2-disc) or a
semi-antipodal (antipodal)  map.

(3) The fixed point set of an orientation preserving $\Z_n$-action
on a closed orientable 3-manifold $M$ is a disjoint union of circles
(may be empty).

(4) In (3) if $M=S^3$ then the fixed point set of an orientation
preserving $\Z_n$-action must be either the empty set or an
unknotted circle.
\end{theorem}

We will also  give a brief recall of orbifold theory for later use,
see \cite{Th}, \cite{BMP}, \cite{MMZ}, \cite{Zi}.

Each orbifold we considered has the form $M/G$, where $M$ is an
$n$-manifold and $G$ is a finite group acting faithfully on $M$. For
a point $x\in M$, denote its stable subgroup by $St(x)$, its image
in $M/G$ by $x'$. If $|St(x)|>1$, $x'$ is called a singular point
and the singular index is $|St(x)|$, otherwise it is called a
regular point. If we forget the singularity we get a topological
space
$|M/G|$ which is called underlying space. 

We can also define the covering space and fundamental group for
orbifold. There is a one to one correspondence between orbifold
covering spaces and conjugate classes of subgroup of orbifold
fundamental group, and regular covering spaces correspond to normal
subgroups.
Following if we say covering spaces or fundamental
groups, it always refer to the orbifold corresponding objects.

A simple  picture we should keep in mind is the following: Suppose
$G$ acts   on $(S^3, \Sigma_g)$. Let $\Gamma=\{x\in S^3|\,\exists
\,g\in G, g\ne \text{id}, s.t.\,gx=x\}$. Then $\Gamma/G$ is the
singular set of the 3-orbifold $S^3/G$, and  $ \Sigma_g/G$ is a
2-orbifold with singular set $ \Sigma_g/G\cap \Gamma/G$.

Suppose a finite cyclic group $G=\Z_n$ acts on $\Sigma_g$. Then
$\Sigma_g/G$ is a 2-orbifold  whose singular set contains isolated
points $a_{1},a_{2},\cdots,a_{k}$, with indices $q_{1}\leq
q_{2}\leq\cdots\leq q_{k}$.  Suppose the genus of $|\Sigma_g/G|$ is
$\hat{g}$. We have the Riemann-Hurwitz formula
$$2-2g=n(2-2\hat{g}-\sum_{i=1}^k (1-\frac{1}{q_i})), \,\, \text{$q_i$ dividing $n$}. \qquad(RH)$$




The following Hurwitz type result  is about the existence and the
classification of the actions of finite group $G$ on $\Sigma_g$.

\begin{theorem}\label{Hur}
(1)A finite  cyclic group $G$ acts on the surface $\Sigma_g$ to get
an orbifold $X=\Sigma_g/G$ if and only if there is an exact sequence
$$1\rightarrow \pi_1(\Sigma_g)\rightarrow \pi_1(X)\rightarrow G\rightarrow 1;$$
(2)If two finite group actions $G$ and $G'$ are conjugate, then
their exact sequences have the following diagram

$$\xymatrix{
  1  \ar[r] & \pi_1(\Sigma_g) \ar[d]_{\cong}\ar[r] & \pi_1(X)\ar[d]_{\cong} \ar[r] & G\ar[d]_{\cong}\ar[r] & 1\\
  1  \ar[r] & \pi_1(\Sigma_g') \ar[r] & \pi_1(X') \ar[r] & G'\ar[r] & 1
  }
$$
(3)If two finite group actions $G$ and $G'$ have the above diagram,
and both actions have no reflection fixed curves, then these actions
are conjugate.
\end{theorem}

\begin{proof}
(1) is parallel to the classical covering space theory, and is the
fundamental property in orbifold theory.

(2) Suppose two actions are conjugate induced by an homeomorphism
between surfaces $\tilde{f}: \Sigma_g \rightarrow \Sigma'_g$, then
there is a diagram about covering maps:
$$\xymatrix{
\Sigma_g \ar[d]_{\tilde{f}}\ar[r] & X\ar[d]_{f} \\
\Sigma_g' \ar[r] & X'
  }
$$
Then $\tilde{f}_*$ and $f_*$ gives the first two vertical
homomorphisms between the fundamental groups. The third one is also
well defined as a quotient of $f_*$.

(3) Suppose the two actions have the group-level diagram, and both
actions have no reflection fixed curves, then we have $f_*:
\pi_1(X)\rightarrow\pi_1(X')$. By \cite[Theorem 5.8.3]{ZVC}, this
homomorphism can be induced by some orbifold homeomorphism $f$. Now
because the left square commutes, $f$ can be lift to a homeomorphism
$\tilde{f}$ between covering surfaces such that it induces the first
vertical homomorphism. And such $\tilde{f}$ gives the desired
conjugation between actions.
\end{proof}

To prove our Theorem \ref{classifacation and extendibility}, besides
quoting above results, we need more results and constructions:

Suppose $h$ is an orientation-reversing periodic map of order $2q$
on compact $p$-manifold $U$, $p=2, 3$. Then the index two subgroup
of $G=\langle h\rangle$ is the unique one $G^o=\langle h^2\rangle$
which acts on $U$ orientation preservingly. Let $X=U/G^o$ be the
corresponding $p$-orbifold, $\pi: U\to X$ be the cyclic branched
covering of degree $q$, and $\langle \bar h\rangle$ be  the induced
order 2 orientation reversing action on $X$.

\begin{lemma}\label{fixed point of h} Under the setting above:
Suppose $x\subset X$ is of index $n$ ($n=1$ if $x$ is a regular
point), and $x$ is a fixed point of $\bar h$. Then

(1) $q/n$ must be odd.

(2) If $p=2$, $x$ must be a regular point of $\bar h$.
\end{lemma}

\begin{proof} (1) Let $N({x})=B^p\subset |X|$ be a $\bar h$-invariant
$p$-ball.  Then $\pi^{-1}(x)=\{\tilde{x_1},.....\tilde{x_l}\}$,
where $l=q/n$, and $\pi^{-1}(B^p)=\{B_1,...,B_l\}$ is a disjoint
union $p$-balls which is an $h$-invariant set.  Moreover the action
of $h$ on $\{B_1,...,B_l\}$ is transitive and orientation reserving
(under the induced orientation.) The stable subgroup of $\tilde
x_1$, $St(\tilde x_1)\subset \Z_{2q}=\langle h\rangle$ is a cyclic
group $\Z_{2n}=\langle h^l\rangle$. Since $St(\tilde x_1)$ contains
orientation reversing element, $l$ must be odd.

(2) In dimension 2, any orientation reversing periodic map must be
conjugated to a reflection about a diameter $L$ of $B^2$ by Theorem
\ref{fixed point set} (1). Therefore any finite cyclic group
containing an orientation reversing element must be $\Z_2$, that is
to say $n=1$ and therefore $x$ is a regular point.
\end{proof}



\begin{lemma}\label{odd is not induced}
Suppose $h$ is an orientation preserving periodic map on $\Sigma_g$
and the number of singular points of $X=\Sigma_g/\langle h \rangle$
is odd. Then $h$ can not extend to $S^3$ in the type $(+,+)$.
\end{lemma}

\begin{proof}
Otherwise let $\tilde h: (S^3,\Sigma_g)\to (S^3,\Sigma_g)$ be a such
extension. As an orientation preserving periodic map on $S^3$, its
fixed point set must be disjoint union of circles by Theorem
\ref{fixed point set} (3), and the singular set $\Gamma_{\langle
\tilde h\rangle}$ of the orbifold $S^3/\langle \tilde h\rangle$ must
be a disjoint union of circles. Then the singular set of the
2-orbifold $\Sigma_g/\langle h\rangle$, as the intersection of those
circles and $|\Sigma_g/\langle h\rangle|$, must has even number of
points, which is a contradiction.
\end{proof}

\begin{lemma}\label{Klein bottle}
The Klein bottle $K$ can not embed into $\mathbb{R}P^3$
\end{lemma}
\begin{proof}
Otherwise there is an embedding $K\subset\mathbb{R}P^3$. We can
assume that $K$ is transversal to some
$\mathbb{R}P^2\subset\mathbb{R}P^3$. Cutting $\mathbb{R}P^3$ along
this  $\mathbb{R}P^2$ we  get the $D^3$, and  $K$ becomes  an
embedded proper surface $S\subset D^3$ with $\chi(S)=\chi(K)=0$.
Note that every embedded proper surface in $D^3$ must be orientable,
so $S$ must be an annulus.   But  the two boundaries of $S$ in
$\partial D^3$ must be identified by the antipodal map on $\partial
D^3$ before the cutting, that is to say we can only get a torus from
$S$, not $K$, which is a contradiction.
\end{proof}

The constructions below provides various extendable periodic maps on
$\Sigma_g$.

\begin{example}\label{cage}
For every $g>1$, we will construct some finite cyclic actions on a
Heegaard splitting $S^3=V_g\cup_{\Sigma_g} V_g'$.

Consider $S^3$ as the unit sphere in $\mathbb{C}^2$,
$$S^3=\{(z_1, z_2)\in\mathbb{C}^2\mid |z_1|^2+|z_2|^2=1\}.$$
Let $$a_m=(e^{\frac{m\pi i}{2}}, 0), m=0, 1, 2, 3.$$
$$b_n=(0, e^{\frac{n\pi i}{g+1}}), n=0, 1, \dots, 2g+1.$$ Connect each
$a_{2l}$ to each $b_{2k}$ with a geodesic in $S^3$ and each
$a_{2l+1}$ to each $b_{2k+1}$ with a geodesic in $S^3$, here $l=0,
1$ and $k=0, 1, \dots, g$. Then we get two two-parted graphs
$\Gamma, \Gamma'\in S^3$ each has $2$ vertices and $g+1$ edges, and
they are in the dual position, see Figure 4. The left one is a
sketch map, the right one presents exactly how the graphs look like.
All the graphs have been projected to $\mathbb{R}^3$ from $S^3$.

\begin{center}
\scalebox{0.5}{\includegraphics{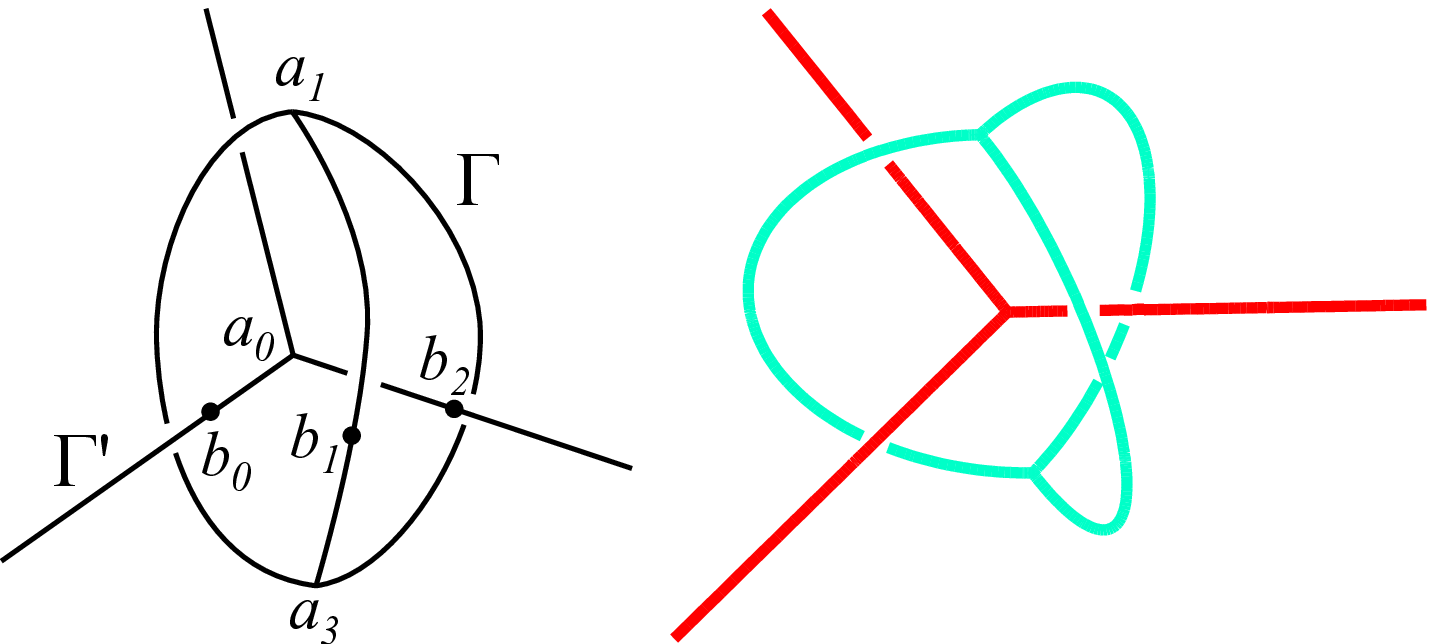}}

Figure 4
\end{center}

We have the following three isometries on $S^3$ which preserves the
graph $\Gamma\cup\Gamma'$:
$$\tau(g): (z_1, z_2)\mapsto(iz_1, e^{\frac{\pi i}{g+1}}z_2),$$
$$\rho: (z_1, z_2)\mapsto(-z_1, z_2),$$
$$\sigma: (z_1, z_2)\mapsto(\bar{z}_1, z_2).$$
Here $\tau$ and $\rho$ preserve the orientation of $S^3$ and
$\sigma$ reverses the orientation of $S^3$. If $g$ is even,
$\langle\tau(g)\rangle$ has order $4(g+1)$.


The points in $S^3$ have equal distance to $\Gamma$ and $\Gamma'$
forms a closed subsurface having genus $g$. This is our $\Sigma_g$.
It cuts $S^3$ into two handlebodies $V_g$ and $V_g'$ which are
neighborhoods of $\Gamma$ and $\Gamma'$. All the isometries
$\tau(g), \rho, \sigma$ preserve $\Gamma\cup\Gamma'$, hence must
also preserve $\Sigma_g$. One can check easily $\tau, \rho, \sigma$
give examples of extendable periodical maps on $(S^3, \Sigma_g)$ of
types $(-,+)$ $(+,+)$, $(-,-)$. In particular $\tau(g)$ gives the a
periodic map on $(S^3, \Sigma_g)$ of type  $(-,+)$ and of order
$4(g+1)$ when $g$ is even. Notice that this is also the maximum
order of cyclic group action on $\Sigma_g$ when $g$ is even. More
concrete and intuitive picture to indicates how $\tau(2)$ acts on
$\Sigma_2$ is given in Figure 5 (for another such description see
\cite{Wa}).

\begin{center}
\scalebox{0.6}{\includegraphics{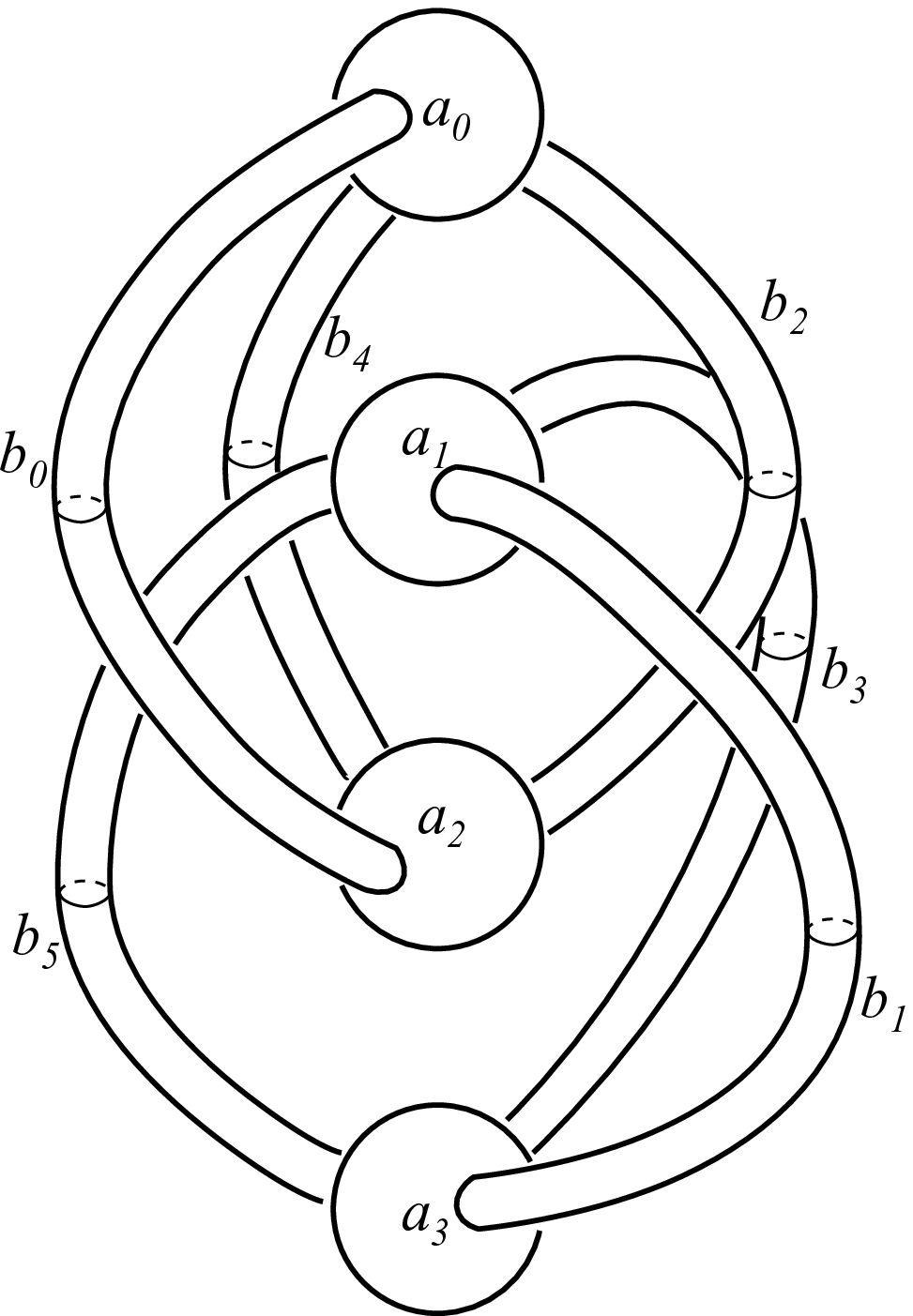}}
$$a_i\mapsto a_{i+1}, b_i\mapsto b_{i+1}$$
Figure 5
\end{center}

Compositions of $\tau(g), \rho, \sigma$ provide extendable
periodical maps on $\Sigma_g$ of required orders and types, say:

(1)  $\sigma\tau(2)$ is a map on $(S^3, \Sigma_2)$ of type $(+,-)$
and order $6$ .

(2)  $\sigma \tau^4(2)$ on $(S^3, \Sigma_2)$ is of type $(-,-)$ and
order 6.



\end{example}

\section{Extendabilities of
periodical maps on $\Sigma_2$.}\label{Proof of the main theorem}

Suppose $\Z_n=\langle h\rangle$ acts on $\Sigma_2$, we know $n\leq
12$. We will discuss all the periods $n\in \{2,3,\dots,12\}$. For
each period $n$, we will divide the discussion into two cases: the
orientation preserving maps on $\Sigma_2$, denote as $(n_+)$; and
orientation reversing maps on $\Sigma_2$, denote as $(n_-)$. For
each case $(n_\epsilon)$, $\epsilon=\pm$, we first discuss the part
(1) of Theorem \ref{classifacation and extendibility}, the
classification of periodic maps; then the part (2) of Theorem
\ref{classifacation and extendibility}, the extendabilities  of
those maps.

We remark that for each odd $n$, the situation is simpler, since all
the possible actions must be orientation-preserving.  If $n=2k$,
then $h$ induces an involution $\bar h$ on $ X =\Sigma_2/\langle
h^2\rangle$, the orbifold corresponding to the unique index two
sub-group $\Z_k$ of $\Z_{2k}$, and if $h$ is orientation-reversing,
$n$ must be $2k$, and  $\bar h$ is orientation-reversing on $ X$.

$(2_+)$ Classification: Now $X=\Sigma_2/\Z_2$ is a closed orientable
2-orbifold with $\chi(X)=\chi(\Sigma_2)/2=-1$ by (RH), and $|X|$
must be a sphere or a torus. Every branched point of $X$ must be
index $2$. So $X$ is either a sphere with six index $2$ branched
points, denoted by $X_1=S^2(2,2,2,2,2,2)$, or a torus with two index
$2$ branched points, denoted by $X_2=T(2,2)$.

By Theorem \ref{Hur} (1), we have an exact sequence $1\rightarrow
\pi_1(\Sigma_2)\rightarrow\pi_1(X_i)\rightarrow\Z_2\rightarrow 1$,
and for each branched point $x$ of $X$, $St(x)$
 must be mapped isomorphically
onto $\Z_2$.

Note $\pi_1(X_1)=\langle a,b,c,d,e,f\mid abcdef=1,
a^2=b^2=c^2=d^2=e^2=f^2=1\rangle,$ the only surjection $\pi_1(X_1)
\to \Z_2=\langle t\mid t^2=1\rangle$ is
$(a,b,c,d,e,f)\mapsto(t,t,t,t,t,t)$, hence this $\Z_2$ action is
unique up to conjugacy, denoted  by $\rho_{2,1}$, whose action on
$\Sigma_2$ is indicated in the left of Figure 6.

Note $\pi_1(X_2)=\langle a,b,u,v\mid aba^{-1}b^{-1}=uv,
u^2=v^2=1\rangle,$ the possible surjections $\pi_1(X_2)\to
\Z_2=\langle t\mid t^2=1\rangle$ are
\begin{equation*} (a,b,u,v)\mapsto
\begin{cases} (1,1,t,t)  \\(1,t,t,t)\\(t,1,t,t)\\(t,t,t,t)
\end{cases}
\end{equation*}
Consider the automorphisms of $\pi_1(X_2):
(a,b,u,v)\mapsto(a,ba,u,v)$, $(a,b,u,v)\mapsto(a,uab,ava^{-1},u)$,
all these representations are equivalent. So this $\Z_2$ action is
unique up to conjugacy, denote it by $\rho_{2,2}$, whose action on
$\Sigma_2$ indicated in the middle and right of Figure 6.

\begin{center}
\scalebox{0.5}{\includegraphics{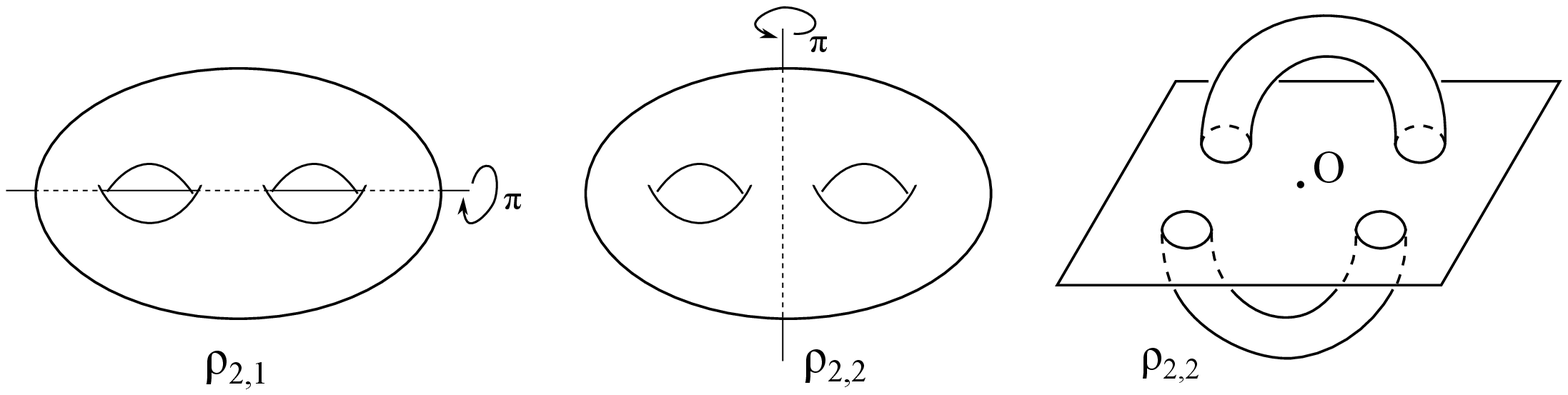}}

Figure 6
\end{center}

Extendibility: As indicated in Figure 6, $\rho_{2,1}$ can extend
orientation-preservingly.

The fixed points set of any orientation-reversing $\Z_2$ action on
$S^3$ is either two points or the 2-sphere  by Theorem \ref{fixed
point set} (2), which can not intersect $\Sigma_2$ with $6$ isolated
points, and therefore $\rho_{2,1}$ can not extend
orientation-reservingly. So we have $\rho_{2,1}\{+\}$.

As indicated in the middle of Figure 6, $\rho_{2,2}$ can extend
orientation-preservingly. If we choose the embedding $\Sigma_2\to
S^3$ as the right-side  of Figure 6,  one can see $\rho_{2,2}$ can
also extend to $S^3$ orientation-reversingly, as a semi-antipodal
map about the origin point $O$ with two fixed points $O$ and
$\infty$. So we have $\rho_{2,2}\{+,-\}$.

$(2_-)$  Classification: Suppose $\tau$ is an order $2$
orientation-reversing homeomorphism of $\Sigma_2$. Consider the
fixed point set $\text{fix}(\tau)$.

If $\text{fix}(\tau)=\varnothing$, then the map $\Sigma_2\rightarrow
\Sigma_2/\Z_2$ is a covering map, and $\Sigma_2/\Z_2$ is a
non-orientable closed surface with $\chi=-1$, which is the connected
sum of a torus and a projective plane. Because the covering map is
unique, the action is unique up to conjugacy, denote it by
$\tau_{2,1}$, whose action on $\Sigma_2$ indicated in Figure 7. Here
it is a semi-antipodal map with fixed points $O$ and $\infty$.

\begin{center}
\scalebox{0.5}{\includegraphics{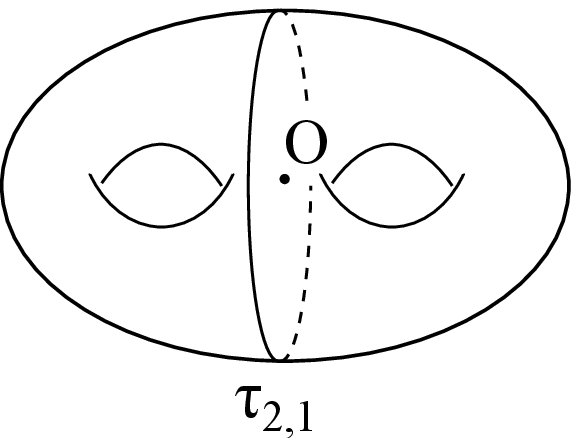}}

Figure 7
\end{center}

Now suppose $\text{fix}(\tau)\neq\varnothing$. Because $\tau$ is
orientation-reversing, $\text{fix}(\tau)$ must be the union of
disjoint circles on $\Sigma_2$. Suppose $\text{fix}(\tau)$ contains
at least one separating circle $C_0$, then $\tau$ changes the two
components of $\Sigma_2\setminus C_0$. So
$\text{fix}(\tau)=\{C_0\}$. In this case the action is also unique,
denote it by $\tau_{2,2}$, whose action on $\Sigma_2$ indicated in
Figure 8. On the left the symbol $\leftrightarrow$ means a
reflection about the middle plane, and on the right is a
$\pi$-rotation.

\begin{center}
\scalebox{0.6}{\includegraphics{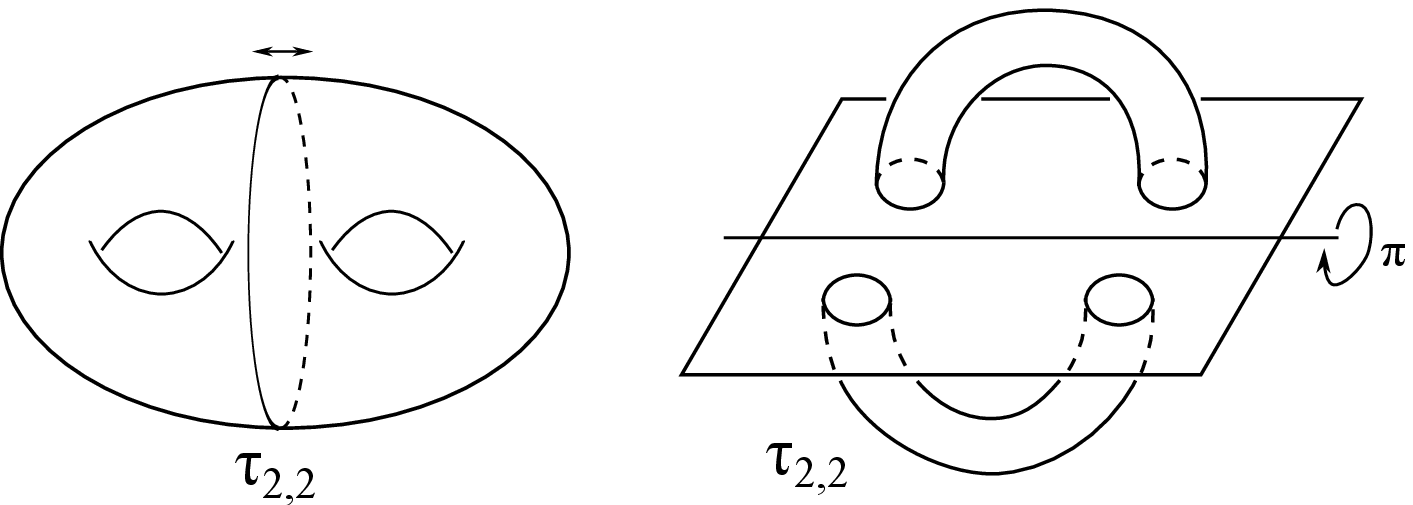}}

Figure 8
\end{center}

Now suppose each component of $\text{fix}(\tau)$ is non-separating.

If $|\text{fix}(\tau)|=1$, say, $\text{fix}(\tau)=\{C_1\}$,
$\Sigma_2\setminus C_1$ is a torus with two holes. And the $\Z_2$
action on $\Sigma_2\setminus C_1$ is fixed point free and changes
the two boundary components, so it induces a double cover to a
non-orientable surface. This is also unique, denoted by
$\tau_{2,3}$, whose action on $\Sigma_2$ indicated in Figure 9.

\begin{center}
\scalebox{0.6}{\includegraphics{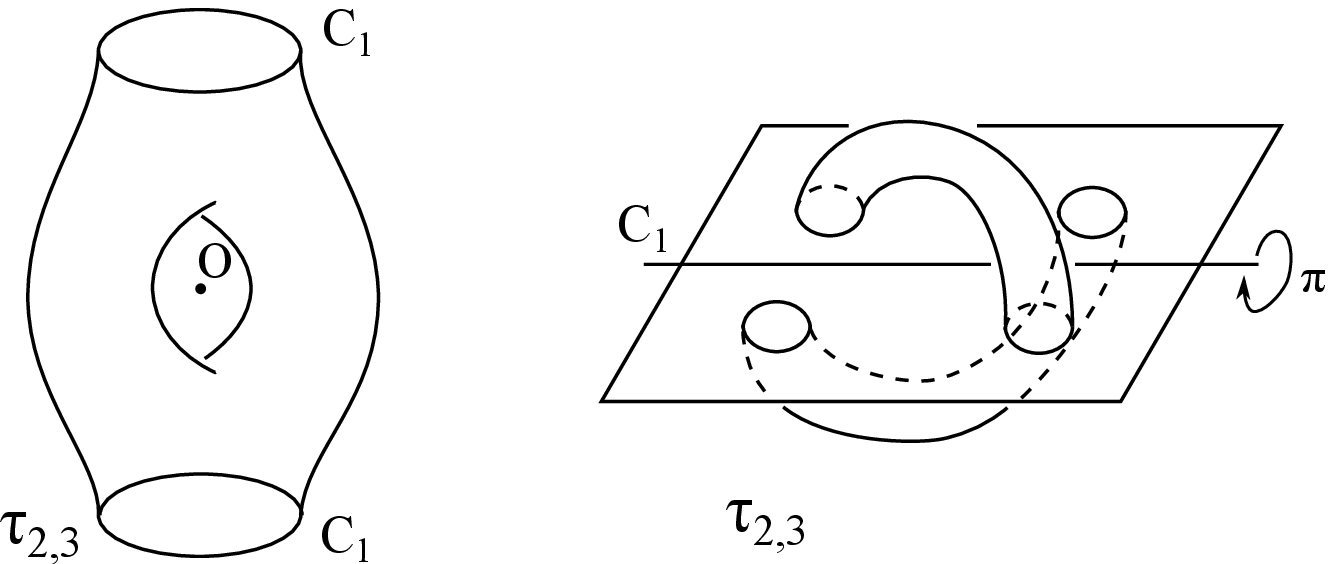}}

Figure 9
\end{center}

\begin{center}
\scalebox{0.6}{\includegraphics{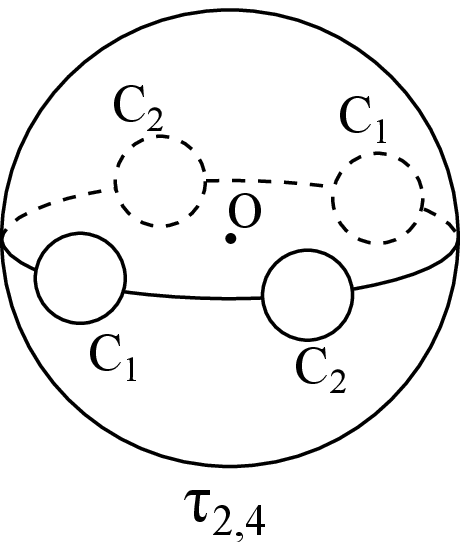}}

\nopagebreak

 Figure 10
\end{center}

If $|\text{fix}(\tau)|=2$, say, $\text{fix}(\tau)=\{C_1,C_2\}$, then
$\Sigma_2\setminus \{C_1,C_2\}$ must be  a sphere with four holes.
And the $\Z_2$ action on $\Sigma_2\setminus \{C_1,C_2\}$ is
fixed-free and changes the four boundary components as two pairs.
The action is also unique, denoted by $\tau_{2,4}$.

If $|\text{fix}(\tau)|=3$, say, $\text{fix}(\tau)=\{C_1,C_2,C_3\}$,
$\Sigma_2\setminus \{C_1,C_2,C_3\}$ are two 3-punctured spheres. The
action is also unique, denoted by $\tau_{2,5}$.

\begin{center}
\scalebox{0.6}{\includegraphics{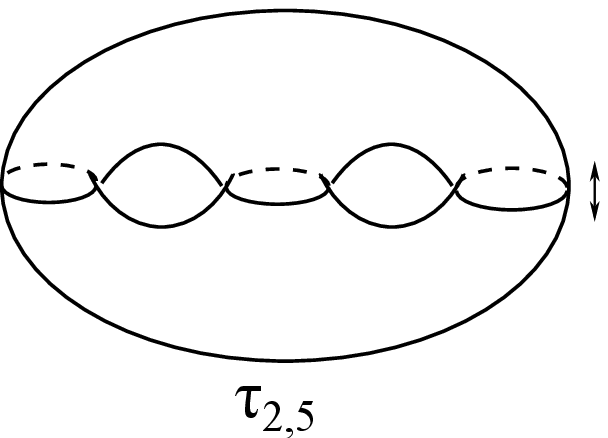}}

Figure 11
\end{center}

If $|\text{fix}(\tau)|\geq 4$, then $\Sigma_2\setminus \{C_i\}$ has
more than two components, which is impossible.

Extendibility: $\tau_{2,1}$ can extend to $S^3$ as a semi-antipodal
map under the embedding  $\Sigma_2\to S^3$ indicated in Figure 7.
Choose an embedding of $\mathbb{R}P^2$ in $\mathbb{R}P^3$, using a
local connected sum with a torus $T$, we get an embedding of $RP^2\#
T$ into $\mathbb{R}P^3$. The double cover of $(\mathbb{R}P^3,RP^2\#
T)$ is $(S^3,\Sigma_2)$. This shows $\tau_{2,1}$ can also extend
orientation-preservingly. So we have $\tau_{2,1}\{+,-\}$.

$\tau_{2,2}$ can obviously extend orientation-reversingly (the left
side of Figure 8). It can also extend orientation-preservingly, as
indicated in right side  of Figure 8. So we have
$\tau_{2,2}\{+,-\}$.

The fixed point set of any orientation-reversing $\Z_2$ action on
$S^3$ is either two points or a 2-sphere  by Theorem  \ref{fixed
point set} (2), and the 2-sphere is separating, which can not
intersect $\Sigma_2$ with a union of non-separating circles. Since
the fixed point set of both $\tau_{2,3}$ and $\tau_{2,4}$ are union
of non-separation circles, neither $\tau_{2,3}$ nor $\tau_{2,4}$ can
extend orientation-reversingly.

The fixed points set of any orientation-preserving $\Z_2$ action on
$S^3$ is either the emptyset or a circle by Theorem \ref{fixed point
set} (4), which can not intersect $\Sigma_2$ with more than one
circle. Since the fixed point set of both $\tau_{2,4}$ and
$\tau_{2,5}$ contain more than one circles, neither $\tau_{2,4}$ nor
$\tau_{2,5}$ can extend orientation-preservingly.

Note $\tau_{2,3}$ can extend orientation-preservingly as in Figure
9. and $\tau_{2,5}$ can extend orientation-reversingly, as shown in
Figure 11.

So we have $\tau_{2,3}\{+\}$, $\tau_{2,4}\{\varnothing\}$, and
$\tau_{2,5}\{-\}$

$(3_+)$ Classification: Here $X=\Sigma_2/\Z_3$ is a closed
orientable 2-orbifold with $\chi(X)=\chi(\Sigma_2)/3=-2/3$, and
every branched point of $X$ must has index $3$. So $X$ is either a
sphere with four index $3$ branched points, denoted by
$X_1=S^2(3,3,3,3)$, or a torus with one index $3$ branched points,
denoted by $X_2=T(3)$.

As above, we have the exact sequence $1\rightarrow
\pi_1(\Sigma_2)\rightarrow\pi_1(X_i)\rightarrow\Z_3\rightarrow 1$,
 and for each branched point $x$ of $X$, $St(x)$
 must be mapped isomorphically
onto $\Z_3$.

For $X_1$, $\pi_1(X_1)=\langle a,b,c,d\mid abcd=1,
a^3=b^3=c^3=d^3=1\rangle,$ up to some permutation of the bases, the
only possible surjections $\pi_1(X_1)\to \Z_3=\langle t\mid
t^3=1\rangle$ is $(a,b,c,d)\mapsto(t,t,t^2,t^2)$, hence this $\Z_3$
action is unique up to conjugacy, denote it by $\rho_3$.

For $X_2$, $\pi_1(X_2)=\langle a,b,u\mid aba^{-1}b^{-1}=u,
u^3=1\rangle,$ from its abelianzation, we know there is no
finite-injective surjection to $\Z_3$, so there is no corresponds
$\Z_3$ action.

Extendibility:  In the embedding $\Sigma_2\in S^3$ of Example 2.5,
one can check directly that
$\Sigma_2/\langle\tau^4_{12}\rangle=S^2(3,3,3,3)$, so $\rho_3$ can
be the restriction of $\tau^4_{12}$, where $\tau_{12}=\tau(2)$ in
Example 2.5. So $\rho_3$ has the extension $\tau^4_{12}$ over $S^3$,
which is type $(+,+)$, and we have $\rho_3\{+\}$.

$(4_+)$ Classification: Here $X=\Sigma_2/\Z_4$ is a closed
orientable 2-orbifold with $\chi(X)=-1/2$. Every branched point of
$X$ has index either $2$ or $4$. So $X$ is either
$X_1=S^2(2,2,2,2,2)$, or $X_2=S^2(2,2,4,4)$, or $X_3=T(2)$. As above
we have an exact sequence $1\rightarrow
\pi_1(\Sigma_2)\rightarrow\pi_1(X_i)\rightarrow\Z_4\rightarrow 1,$
 and
for each branched point $x$ of $X$, $St(x)$ must inject into $\Z_4$.

For $X_1$, $\pi_1(X_1)=\langle a,b,c,d,e\mid abcde=1,
a^2=b^2=c^2=d^2=e^2=1\rangle,$ each generator is corresponding to
some branched point, and must be mapped to $t^2$ in $\Z_4=\langle
t\mid t^4=1\rangle$, and it is impossible, so there is no
corresponding $\Z_4$ action.

For $X_3$, $\pi_1(X_3)=\langle a,b,u\mid aba^{-1}b^{-1}=u,
u^2=1\rangle,$ $u$ must be mapped to $t^2$ in $\Z_4=\langle t\mid
t^4=1\rangle$, and it is impossible, so there is no corresponding
$\Z_4$ action.

For $X_2$, $\pi_1(X_2)=\langle a,b,x,y\mid abxy=1,
a^2=b^2=x^4=y^4=1\rangle,$ up to some permutation of the bases, the
only possible representation from $\pi_1(X_2)$ to $\Z_4=\langle
t\mid t^4=1\rangle$ is $(a,b,x,y)\mapsto(t^2,t^2,t,t^3)$, hence this
$\Z_4$ action is unique up to equivalent, denote it by $\rho_4$.

\begin{center}
\scalebox{0.6}{\includegraphics{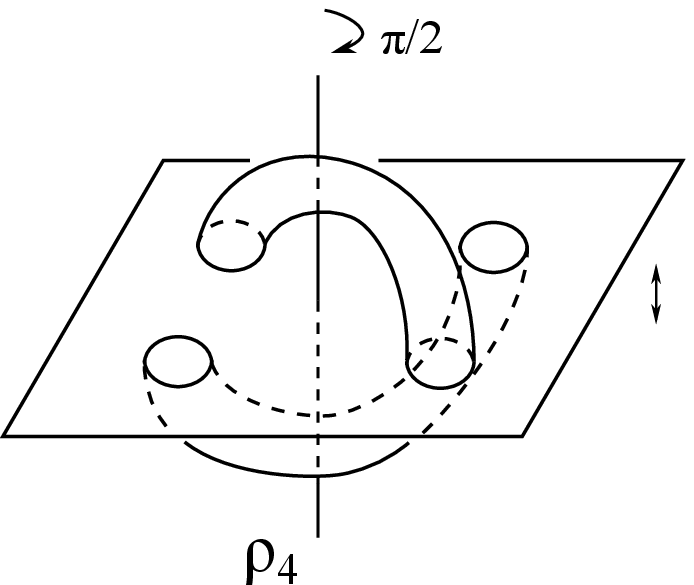}}

Figure 12
\end{center}

Extendibility: From Figure 12, $\rho_4$ can extend
orientation-reversingly as a $\pi/4$-rotation together with a
reflection. Note $\rho_4$ has fixed point. Suppose $\rho_4$ can
extend orientation-preservingly, then by Lemma \ref{fixed point set}
(4), the singular set of $S^3/\rho_4$ must be a circle of index 4,
therefore the index of singular points of $X_2$ must be also 4,
which contradicts that $X_2=S^2(2,2,4,4)$. So we have $\rho_4\{-\}$.

$(4_-)$ Classification: Consider the orbifolds $ X =\Sigma_2/\langle
h^2\rangle$. From the discussion in $(2_+$), either
$X_1=S^2(2,2,2,2,2,2)$, or $X_2=T^2(2,2)$. By Lemma \ref{fixed point
of h} (1),  the orientation-reversing involution $\bar h_i$ on $X_i$
has no regular fixed point, so the orbifold $X_i/\Z_2=\Sigma_2/\Z_4$
is either a projective plane with three index $2$ branched points,
denoted by $Y_1=\mathbb{R}P^2(2,2,2)$, or a Klein bottle $K$ with
one index $2$ branched point, denoted by $Y_2=K(2)$.

For $Y_1$, $\pi_1(Y_1)=\langle x,a,b,c\mid abc=x^2,
a^2=b^2=c^2=1\rangle,$ up to some permutation of the bases, the
possible representations from $\pi_1(Y_1)$ to $\Z_4=\langle t\mid
t^4=1\rangle$ are
\begin{equation*} (x,a,b,c)\mapsto
\begin{cases} (t,t^2,t^2,t^2)  \\(t^3,t^2,t^2,t^2)
\end{cases}
\end{equation*}
and an automorphism of $\Z_4$: $t\mapsto t^3$ changes the two
representations. Hence this $\Z_4$ action is unique up to
equivalent, denote it by $\tau_{4,1}$.

For $Y_2$, $\pi_1(Y_2)=\langle x,a,b\mid aba^{-1}b=x, x^2=1\rangle,$
the possible representations from $\pi_1(Y_2)$ to $\Z_4=\langle
t\mid t^4=1\rangle$ are
\begin{equation*} (x,a,b)\mapsto
\begin{cases} (t^2,*,t)  \\(t^2,*,t^3)
\end{cases}
\end{equation*}
here $*$ means $a$ may be mapped to any element in $\Z_4$. Consider
the automorphism of $\pi_1(Y_2): (x,a,b)\mapsto(x,ab,b)$ and some
automorphism of $\Z_4$, all these representations are equivalent.
Hence this $\Z_4$ action is unique up to equivalent, denote it by
$\tau_{4,2}$, whose action on $\Sigma_2$ indicated in Figure 13.

\begin{center}
\scalebox{0.6}{\includegraphics{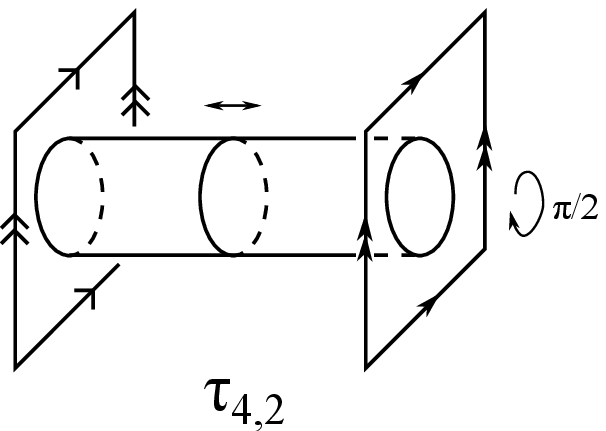}}

Figure 13
\end{center}

Extendibility: In the embedding $\Sigma_2\in S^3$  of Example 2.1,
one can check directly that
$\Sigma_2/\langle\rho\tau^3_{12}\rangle=RP^2(2,2,2)$, so
$\tau_{4,1}$ can be the restriction of $\tau^3_{12}$, therefore has
the extension $\tau^3_{12}$ over $S^3$, which is of type $(-,+)$. It
can also extend orientation-reversingly: see Figure 14. So we have
$\tau_{4,1}\{+,-\}$.

\begin{center}
\scalebox{0.6}{\includegraphics{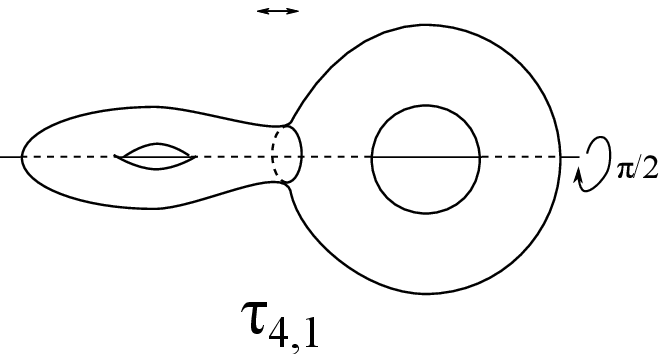}}

Figure 14
\end{center}


$\tau_{4,2}$ can not extend orientation-preservingly: otherwise,
there will be an embedding of Klein bottle $K$ into $|S^3/\Z_4|$.
Here $S^3/\Z_4$ has branched points, it can not be a $\Z_4$-Len
space, so $|S^3/\Z_4|$ must be $S^3$ or $\mathbb{R}P^3$. But $K$ can
not embed into either $S^3$ or $\mathbb{R}P^3$ by Lemma \ref{Klein
bottle}.

$\tau_{4,2}$ can not extend orientation-reversingly: by applying
Dehn's Lemma, the orbifold $X_2$  must bound a $3$-orbifold $\Theta$
in $S^3/\Z_2$, with $|\Theta|$ a solid torus. There is a proper
branched arc $L$ of index 2 in $\Theta$. It is easy to see
$\chi(|\Theta\setminus N(L)|)=-1$, so  $\bar h_2|$ on
$|\Theta\setminus N(L)|$ must have fixed points. This means $\bar
h_2$ on $\Theta$ must have regular fixed points, which contradicts
to Lemma \ref{fixed point of h} (1). So we have
$\tau_{4,2}\{\varnothing\}$.

($5_+$) Now $X=\Sigma_2/\Z_5$ is a closed orientable 2-orbifold,
$\chi(X)=\chi(\Sigma_2)/5=-2/5$, and every branched point of $X$
must has index $5$. So $X=S^2(5,5,5)$. As above, we have an exact
sequence $1\rightarrow
\pi_1(\Sigma_2)\rightarrow\pi_1(X)\rightarrow\Z_5\rightarrow 1$,
 and
for each branched point $x$ of $X$, $St(x)$ must be mapped
isomorphically onto $\Z_5$.

Now $\pi_1(X)=\langle a,b,c\mid abc=1, a^5=b^5=c^5=1\rangle,$ up to
some permutation of the bases, the possible surjections  $\pi_1(x)
\to \Z_5=\langle t\mid t^5=1\rangle$ are
\begin{equation*} (a,b,c)\mapsto
\begin{cases} (t,t,t^3)  \\(t,t^2,t^2)\\(t^3,t^3,t^4)\\(t^2,t^4,t^4)
\end{cases}
\end{equation*}
Consider the automorphism of $\Z_5:t\mapsto t^2$ and $t\mapsto t^4$,
all these representations are conjugate. Hence this $\Z_5$ action is
unique up to conjugacy, denote it by $\rho_5$, whose action on
$\Sigma_2$ was indicated in Figure 21.

Extendibility: Suppose $\rho_5$ is extendable, then the extension on
$(S^3, \Sigma_2)$ must be type $(+,+)$, which contradicts to Lemma
\ref{odd is not induced}, since the orbifold $\Sigma_2/\langle\rho_5
\rangle=S^2(5,5,5)$ contains three singular points. So we have
$\rho_5\{\varnothing\}$.

($6_+$) Classification:  The orbifolds correspond to the index-$2$
subgroup must be $X=(S^2;3,3,3,3)$, as we see in $(3_+)$. Then
orientation-preserving $\Z_2$ actions on $X$ give the  orbifold
$X/\Z_2=\Sigma_2/\Z_6$, which is either $Y_1=S^2(2,2,3,3)$, or
$Y_2=S^2(3,6,6)$.

For $Y_1$, $\pi_1(Y_1)=\langle a,b,x,y\mid abxy=1,
a^2=b^2=x^3=y^3=1\rangle,$ up to some permutation of the bases, the
only possible representation from $\pi_1(Y_1)$ to $\Z_6=\langle
t\mid t^6=1\rangle$ is $(a,b,x,y)\mapsto(t^3,t^3,t^2,t^4)$, hence
this $\Z_6$ action is unique up to equivalent, denote it by
$\rho_{6,1}$, whose action on $\Sigma_2$ indicated in Figure 15.

\begin{center}
\scalebox{0.8}{\includegraphics{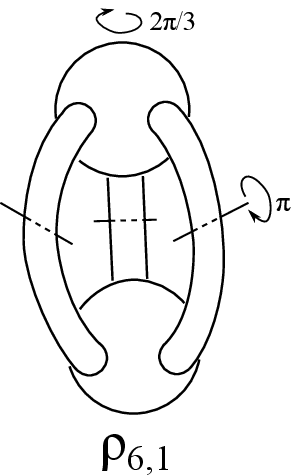}}

Figure 15
\end{center}

For $Y_2$, $\pi_1(Y_2)=\langle a,b,x\mid abx=1,
a^6=b^6=x^3=1\rangle,$ the possible representations from
$\pi_1(Y_2)$ to $\Z_6=\langle t\mid t^6=1\rangle$ are
\begin{equation*} (a,b,x)\mapsto
\begin{cases} (t,t,t^4)  \\(t^5,t^5,t^2)
\end{cases}
\end{equation*}
Consider the automorphism of $\Z_6$: $t\mapsto t^5$, these
representations are equivalent. Hence this $\Z_6$ action is unique
up to conjugacy, denote it by $\rho_{6,2}$.

Extendibility: In the embedding $\Sigma_2\subset S^3$  in Example
2.1, one can check directly that
$\Sigma_2/\langle\tau^2_{12}\rangle=S^2(2,2,3,3)$, so $\rho_{6,1}$
can be the restriction of $\tau^2_{12}$, therefore has the extension
$\tau^2_{12}$ over $S^3$, which is of type $(+,+)$.

$\rho_{6,1}$ can not extend orientation-reversingly: Otherwise, the
extension must be of type $(+,-)$, which must interchanges two
components of $S^3\setminus \Sigma_2$. Denote by $\Theta_1$ and
$\Theta_2$ the two 3-orbifold bounded by $X=S^2(3,3,3,3)$, and by
$\{A,B,C,D\}$ the  four branched points on $X$. By apply Smith
theory, we may suppose that two branched arcs in $\Theta_1$ are $AB$
and $CD$, and two branched arcs in $\Theta_2$ are $BC$ and $DA$ ,
see Figure 16. Note the induced involution $\bar \rho_{6,1}$ on $X$
is a $\pi$-rotation about two ordinary points, and  interchanges
$\Theta_1$ and $\Theta_2$. So $\bar \rho_{6,1}(A)\ne A$. If $\bar
\rho_{6,1}$ interchanges $A$ and $B$, $\bar \rho_{6,1}$ will keep
the singular arc $AB$ invariant;  if $\bar \rho_{6,1}$ interchanges
the pairs $(A, B)$ and the pair $(C,D)$, then $\bar \rho_{6,1}$
interchanges the singular arcs $AB$ and $CD$. In either case we
would have $\bar \rho_{6,1}(\Theta_1)=\Theta_1$, this is a
contradiction.
 So we have $\rho_{6,1}\{+\}$.

\begin{center}
\scalebox{0.6}{\includegraphics{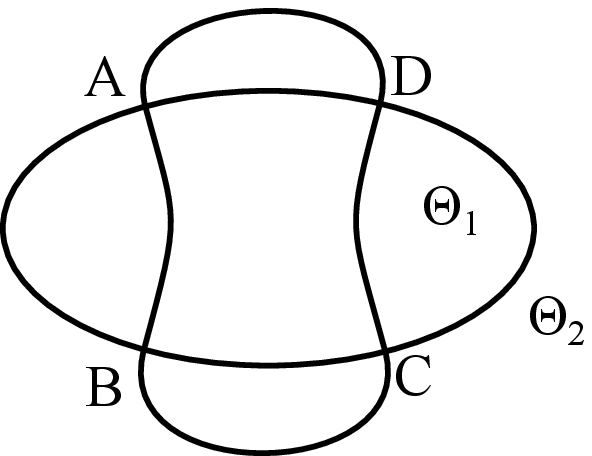}}

Figure 16
\end{center}

Since $\Sigma_2/\langle\rho_{6,2}\rangle=S^2(3,6, 6)$ has three
singular points, $\rho_{6,2}$ can not extend to $S^3$ in type
$(+,+)$ by Lemma \ref{odd is not induced}. In the embedding
$\Sigma_2\in S^3$ in Example 2.1, one can check that
$\Sigma_2/\langle\sigma\tau_{12}\rangle=S^2(3,6,6)$, so $\rho_{6,2}$
can be the restriction of $\sigma\tau_{12}$, therefore has the
extension $\sigma\tau_{12}$ over $S^3$ which is of type $(+,-)$. So
we have $\rho_{6,2}\{-\}$.

$(6_-)$ Classification: Consider the orbifolds $ X =\Sigma_2/\langle
h^2\rangle=(S^2;3,3,3,3)$. The action of $\bar h$ on $X$ is either
the antipodal map, corresponding to the orbifold
$Y_1=\mathbb{R}P^2(3,3)$, or a reflection about circle which
containing no branched points, corresponding to the orbifold
$Y_2=\bar D^2(3,3)$: a disk with reflection boundary and branched
points $(3,3)$.

For $Y_1$, $\pi_1(Y_1)=\langle a,b,x\mid ab=x^2, a^3=b^3=1\rangle,$
the possible representations from $\pi_1(Y_1)$ to $\Z_6=\langle
t\mid t^6=1\rangle$ are
\begin{equation*} (a,b,x)\mapsto
\begin{cases} (t^2,t^2,t^2)  \\(t^4,t^4,t^4)\\(t^2,t^4,t^3)
\end{cases}
\end{equation*}

But the first two are not surjective, so only the third one is
possible. Hence this $\Z_6$ action is unique up to equivalent,
denote it by $\tau_{6,1}$. An illustration of  the  action
$\tau_{6,1}$ on $\Sigma_2$ is based on Figure 17: denote the union
of three tubes, and its top and bottom boundary components by $A$,
$\partial A_+$ and $\partial A_-$; two 3-punctured 2-spheres by
$S_+$ and $S_-$. Then
$$\Sigma_2=S_+\cup_{\partial S_+=\partial A_+} A \cup_{\partial S_-=\partial A_-} \cup S_-,$$
where the identification ${\partial S_+=\partial A_+}$ is given in
the most obviously way, the identification $\partial S_-=\partial
A_-$ and that of ${\partial S_+=\partial A_+}$ is differ by $\pi$.
Then $\tau_{6,1}$ restricted on each tube is an antipodal map, and
$\tau_{6,1}$ restricted on $S_+\cup S_-$ is a reflection about the
plane between them.

\begin{center}
\scalebox{0.8}{\includegraphics{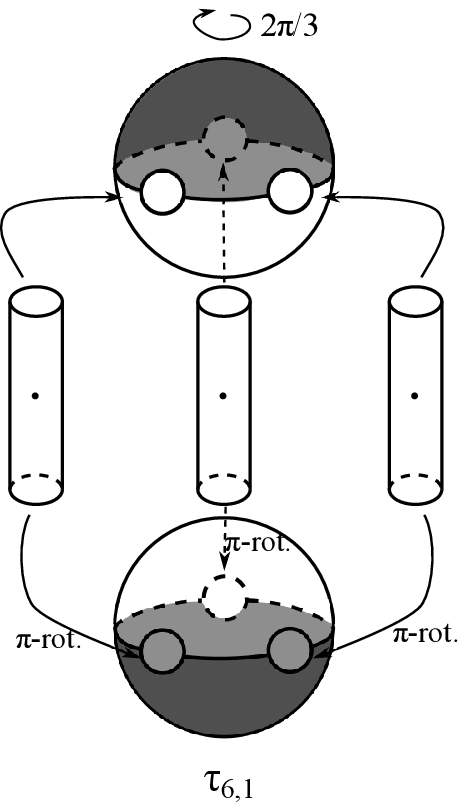}}

Figure 17
\end{center}

For $Y_2$, $\pi_1(Y_2)=\langle a,b,x,r\mid ab=x,r^2=1,xr=rx,
a^3=b^3=1\rangle,$ the possible representations from $\pi_1(Y_1)$ to
$\Z_6=\langle t\mid t^6=1\rangle$ are
\begin{equation*} (a,b,x,r)\mapsto
\begin{cases} (t^2,t^4,1,t^3)\\(t^2,t^2,t^4,t^3)  \\(t^4,t^4,t^2,t^3)
\end{cases}
\end{equation*}
Denote the first one by $\tau_{6,2}$. The second one and the third
one are equivalent by a automorphism of $\Z_6$: $t\mapsto t^5$,
denote this one by $\tau_{6,3}$, whose action on $\Sigma_2$
indicated in Figure 18: joining two disks with $3$ bands, each with
a half twist, then we get a surface $F$, which can be view  as the
Seifert surface of the trefoil knot. Its neighborhood in $S^3$ is a
$V_2\cong F\times [-1,1]$. Then the action of $\tau_{6,3}$ on
$\Sigma_2-\partial V_2$ is the composition of a $2\pi/3$ rotation
and a reflection about $F\times \{0\}$.

\begin{center}
\scalebox{0.7}{\includegraphics{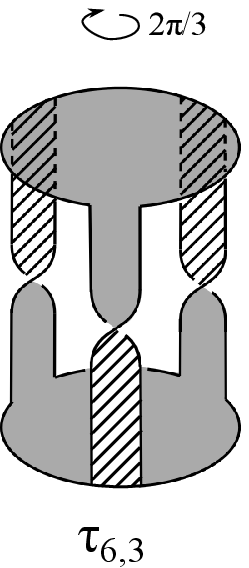}}

Figure 18
\end{center}

Extendibility:  Just applying similar proof for the non-existence of
$\rho_{6,1}\{-\}$, we can show that
 $\tau_{6,1}$, $\tau_{6,2}$ and $\tau_{6,3}$ can not extend
orientation-preservingly.

$\tau_{6,1}$ can not extend orientation-reversingly: Otherwise
$\tau_{6,1}$ acts on $(S^3, \Sigma_2)$ in the type of $(-,-)$. Since
$Y_1=\mathbb{R}P^2(3,3)$, the $\Z_2$-action $\tau_{6,1}^3$ has no
fixed point on $\Sigma_2$, hence $fix(\tau_{6,1}^3)$  on $S^3$ is
two pints $\{x,y\}$  by Theorem \ref{fixed point set}. Since
$\tau_{6,1}^2(x)$ must be a fixed point of $\tau_{6,1}^3$ and
$\tau_{6,1}^2$ is of order 3, we must have $\tau_{6,1}^2(x)=x$, then
$\tau_{6,1}(x)=x$. Hence the branched set
$S^3/\langle\tau_{6,1}^2\rangle$ is  a circle $C$ of index 3 and
$|S^3/\langle\tau_{6,1}^2\rangle|=S^3$. Denote by $\Theta$ a
3-orbifold bounded by $X= S^2(3,3,3,3)$ containing $\bar x$, the
image of $x$. Then $|\Theta|=D^3$, $\Theta\cap C$ in two branched
arcs of index 3. But the orientation revering involution $\bar
\tau_{6,1}$ on $(\Theta, X)$ is the antipodal map, hence $\bar x$ is
the only fixed point of $\bar \tau_{6,1}$ on $\Theta$. Clearly $\bar
x\in \Theta\cap \bar C$, so $\bar \tau_{6,1}$ keeps each  arc of
$\Theta\cap \bar C$ invariant, and therefore $\bar \tau_{6,1}$ on
$\Theta$ has at least two fixed points, this is a contradiction. So
we have $\tau_{6,1}\{\varnothing\}$.

In the embedding $\Sigma_2\in S^3$ in Example 2.1, one can check
that $\Sigma_2/\langle\sigma\tau_{12}^4\rangle=\bar D^2(3,3)$, so
$\tau_{6,2}$ can be the restriction of $\sigma\tau_{12}^4$,
therefore has the extension $\sigma\tau_{12}^4$ over $S^3$ which is
of type $(-,-)$. So we have $\tau_{6,2}\{-\}$.

$\tau_{6,3}$ can not extend orientation-reversingly. Otherwise
$\tau_{6,3}$ acts on $(S^3, \Sigma_2)$ in the type of $(-,-)$. Now
$fix(\tau_{6,3}^3)\cap \Sigma_2$ is a separating circle $C$. Now
$fix(\tau_{6,3}^3)=S^2\subset S^3$. Let $\Sigma_2\backslash
C=\Sigma_{+}\cup \Sigma_{-}$ and $S^2\backslash C=D_{+}\cup D_{-}$.
So the $\Z_3$-action $\tau_{6,3}^2$  on $\Sigma_{+}\cup D_{+}\cong
T$ is extendable extendable action. But the orbifold
$T/\Z_3=(S^2;3,3,3)$, which can not embed in
$S^3/\langle\tau_{6,3}^2\rangle$ by Lemma \ref{odd is not induced}.
So we have $\tau_{6,3}\{\varnothing\}$.

$(7_+)$ Classification: Now $X=\Sigma_2/\Z_7$ is a closed orientable
2-orbifold with $\chi(X)=\chi(\Sigma_2)/7=-2/7$, and every branched
point of $X$ must has index $7$. There is no such orbifold.

$(8_+)$ Classification: Consider the orbifolds $ X =\Sigma_2/\langle
h^2\rangle=(S^2;2,2,4,4)$. Let $Y=\Sigma_2/\Z_8=X/\Z_2$, then $Y$ is
either $Y_1=(S^2;2,2,2,4)$, or $Y_2=(S^2;4,4,4)$ or
$Y_3=(S^2;2,8,8)$.

For $Y_1$, $\pi_1(Y_1)=\langle a,b,c,x\mid abcx=1,
a^2=b^2=c^2=x^4=1\rangle,$ there is no surjection $\pi_1(Y_1)\to
\Z_8$, so there is no corresponding $\Z_8$ action.

For $Y_2$, $\pi_1(Y_2)=\langle a,b,c\mid abc=1,
a^4=b^4=c^4=1\rangle,$ there is no surjection $\pi_1(Y_2)\to \Z_8$
so there is no corresponding $\Z_8$ action.

For $Y_3$, $\pi_1(Y_3)=\langle a,b,x\mid abx=1,
a^8=b^8=x^2=1\rangle,$ the possible surjections $\pi_1(Y_3)\to
\Z_8=\langle t\mid t^8=1\rangle$ are
\begin{equation*} (a,b,x)\mapsto
\begin{cases} (t,t^3,t^4)  \\(t^5,t^7,t^4)
\end{cases}
\end{equation*}
Consider the automorphism of $\Z_8$: $t\mapsto t^5$, these
representations are equivalent. Hence this $\Z_8$ action is unique
up to equivalent, denote it by $\rho_8$, whose action on $\Sigma_2$
indicated in Figure 19.

\begin{center}
\scalebox{0.6}{\includegraphics{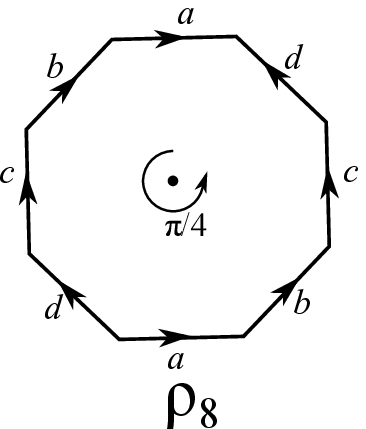}}

Figure 19
\end{center}

Extendibility: Note $\rho_8^2=\rho_4$. If either $\rho_8(+)$ or
$\rho_8(-)$ exists, we have $\rho_4(+)$, which contradicts that
$\rho_4\{\varnothing\}$. So we have $\rho_8\{\varnothing\}$.

$(8_-)$ Classification:  Consider the orbifolds $ X
=\Sigma_2/\langle h^2\rangle=(S^2;2,2,4,4)$. By Lemma 2.3 (1), $\bar
h$ on $X$ has no regular fixed points on $X$, so it must be a
antipodal map, therefore $Y=X/\Z_2$ must be $(\mathbb{R}P^2;2,4)$.
$\pi_1(Y)=\langle a,b,x\mid ab=x^2, a^2=b^4=1\rangle,$ the possible
representations from $\pi_1(Y)$ to $\Z_8=\langle t\mid t^8=1\rangle$
are
\begin{equation*} (a,b,x)\mapsto
\begin{cases} (t^4,t^2,t^3)  \\(t^4,t^2,t^7)\\(t^4,t^6,t^5)\\(t^4,t^6,t)
\end{cases}
\end{equation*}
Consider the automorphism of $\Z_8:t\mapsto t^5$ and $t\mapsto t^7$,
all these representations are equivalent. Hence this $\Z_8$ action
is unique up to equivalent, denote it by $\tau_8$, whose action on
$\Sigma_2$ indicated in Figure 20.

\begin{center}
\scalebox{0.6}{\includegraphics{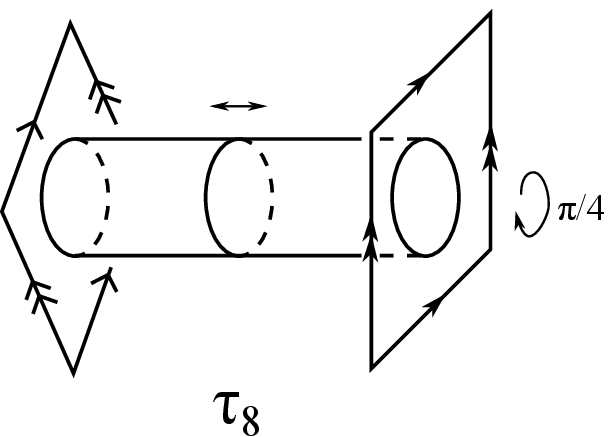}}

Figure 20
\end{center}

Extendibility: Still we have $\tau_8^2=\rho_4$. The same reason used
in $(8_+)$ shows $\tau_8\{\varnothing\}$.

$(9_+)$ Classification: First consider the $\Z_3$ subgroup orbifold
$X$. As we see in  $3_+$, $X=S^2(3,3,3,3)$. The only possible
orbifold $Y=X/\Z_3$ is $S^2(3,3,9)$. But its fundamental group can
not surjectively map onto $\Z_9$. So there is no such an action.

$(10_+)$ Classification: As we see in $(5_+)$, the orbifolds $ X
=\Sigma_2/\langle h^2\rangle=S^2(5,5,5)$. Let
$Y=\Sigma_2/\Z_{10}=X/\Z_2$, then $Y=S^2(2,5,10)$. $\pi_1(Y)=\langle
a,b,c\mid abc=1, a^2=b^5=c^{10}=1\rangle,$ the possible
representations from $\pi_1(Y)$ to $\Z_{10}=\langle t\mid
t^{10}=1\rangle$ are
\begin{equation*} (a,b,c)\mapsto
\begin{cases} (t^5,t^2,t^3)  \\(t^5,t^4,t)\\(t^5,t^6,t^9)\\(t^5,t^8,t^7)
\end{cases}
\end{equation*}
Consider the automorphism of $\Z_{10}:t\mapsto t^3$ and $t\mapsto
t^7$, all these representations are equivalent. Hence this $\Z_{10}$
action is unique up to conjugacy, denote it by $\rho_{10}$,  whose
action on $\Sigma_2$ was indicated in Figure 21. (See \cite{Wa} for
detailed description).

\begin{center}
\scalebox{0.6}{\includegraphics{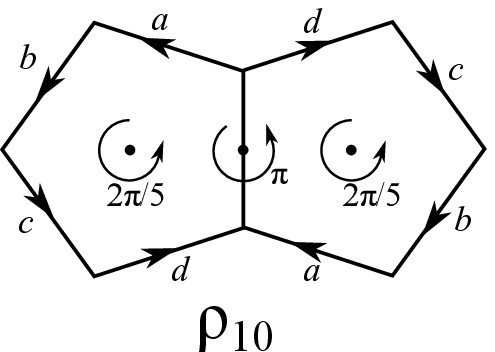}}
\nopagebreak

Figure 21
\end{center}

Extendibility: Since $\rho_{10}^2=\rho_5$, The same reason used in
$(8_+)$ shows  $\rho_{10}\{\varnothing\}$.

$(10_-)$ Classification: Consider the orbifolds $ X
=\Sigma_2/\langle h^2\rangle=S^2(5,5,5)$. By symmetry consideration
, $\bar h$ can not be a antipodal map, since there are exactly three
singular point of index 5, and $\bar h$ also can not be  a
reflection about a circle, since such a circle must passing a
branched points  of index 5, which contradicts to Lemma \ref{fixed
point of h} (2).

$(11_+)$ Classification:  Now $X=\Sigma_2/\Z_{11}$ is a closed
orientable 2-orbifold $\chi(X)=\chi(\Sigma_2)/11=-2/11$, and every
branched point of $X$ must has index $11$. There is no such
orbifold.

$(12_+)$ Classification: First consider the $\Z_6$ subgroup orbifold
$X$. $X$ is either $X_1=S^2(2,2,3,3)$ or $X_2=S^2(3,6,6)$ as wee see
in $(6_+)$. So $Y=X/\Z_2$ is either  $Y_1=(S^2;3,4,4)$, or
$Y_2=(S^2;2,6,6)$, or $Y_3=(S^2;2,2,2,3)$, none of these fundamental
groups has surjection  onto $\Z_{12}$, so there is no
orientation-preserving actions of $Z_{12}$.

$(12_-)$ Classification: Still the $\Z_6$ subgroup orbifold will be
either  $X_1=(S^2;2,2,3,3)$ or $X_2=(S^2;3,6,6)$. The
orientation-reversing $\Z_2$-action on $X$ can not have regular
fixed points, so it must be a antipodal map. So only $X_1$ is
possible and $Y=X_1/\Z_2=(\mathbb{R}P^2;2,3)$. $\pi_1(Y)=\langle
a,b,x\mid ab=x^2, a^2=b^3=1\rangle,$ the possible representations
from $\pi_1(Y)$ to $\Z_{12}=\langle t\mid t^{12}=1\rangle$ are
\begin{equation*} (a,b,x)\mapsto
\begin{cases} (t^6,t^4,t^5)  \\(t^6,t^4,t^{11})\\(t^6,t^8,t^7)\\(t^6,t^8,t)
\end{cases}
\end{equation*}
Consider the automorphism of $\Z_{12}:t\mapsto t^5$ and $t\mapsto
t^7$, all these representations are equivalent. Hence this $\Z_{12}$
action is unique up to equivalent, which is the $\tau_{12}$ in
Example 2.1.

Extendibility: $\tau_{12}$ can not extend orientation reversely:
Otherwise $\tau_{12}$ acts on each component of $S^3\setminus
\Sigma_2$. Denote them by $\Theta_1$ and $\Theta_2$ the two
3-orbifold bounded by $X=(S^2;2,2,3,3)$, then each $\Theta_i$ has
two singular arcs of index 2 and index 3 respectively. We may assume
$|\Theta_1|=B^3$ and the induced orientation reversing involution
$\bar \tau_{12}$ acts on each $\Theta_i$. Hence  $\bar \tau_{12}$ on
$|\Theta_1|$ must be a reflection about an equator disc, which has
regular fixed point, which contradicts to Lemma \ref{fixed point of
h} (1). From Example 2.5, we have $\tau_{12}\{+\}$.\qed

\bigskip\noindent\textbf{Acknowledgement}. The second author is supported by Beijing International Center for Mathematical Research£¬ Peking University.
The third author is partially supported by grant No.10631060 of the
National Natural Science Foundation of China.

\end{document}